\documentclass[11pt]{amsart}
\usepackage{amsfonts}
\usepackage{amscd,amssymb,float}

\newtheorem{theorem}{Theorem}
\newtheorem{lm}{Lemma}
\newtheorem{prop}{Proposition}

\newtheorem{re}{Remark}

\numberwithin{equation}{section} \numberwithin{figure}{section}

\input{tcilatex}

\begin{document}
\title{\textbf{Convergence of Einstein Yang-Mills Systems}}
\author{ Hongliang Shao }
\date{}
\maketitle

\begin{abstract}
In this paper, we prove a convergence theorem for sequences of
Einstein Yang-Mills systems on $U(1) $-bundles over  closed
$n$-manifolds with some bounds for volumes, diameters, $L^{2}$-norms
of  bundle curvatures  and  $L^{\frac{n}{2}}$-norms of  curvature
tensors.  This result is a generalization of earlier compactness
theorems for Einstein manifolds.
\end{abstract}

\section{Introduction}

A Riemannian metric $g$ on a smooth manifold $M$ is called an Einstein
metric with Einstein constant $\lambda $, if $g$ has constant Ricci
curvature $\lambda $, i.e.,
\begin{equation*}
\mathrm{Ric}=\lambda g
\end{equation*}%
(cf. \cite{Be}). Einstein metrics are considered as the nicest
metrics on manifolds. Nonetheless, it is well-known that some
manifolds wouldn't
admit any Einstein metric (cf. \cite{Be}, \cite{Hi}).\ In \cite{St} and \cite%
{Yo}, the notion of Einstein Yang-Mills system is introduced as a
generalization of Einstein metrics by coupling Einstein equations with
Yang-Mills equations. Let $\mathcal{L}$ be a principal $U(1)$-bundle over a
smooth manifold $M$. If a Riemannian metric $g$ on $M$ and a connection $A$
of $\mathcal{L}$ satisfy the following equations
\begin{equation}
\left\{
\begin{array}{c}
\mathrm{Ric}-\frac{1}{2}\eta =\lambda g \\
\text{ \ \ \ \ \ \ \ }d^{\ast }F=0\text{ ,\ }%
\end{array}%
\right.  \tag{1}  \label{eq:1}
\end{equation}%
where $F$ is the curvature of $A$ and $\eta _{ij}=g^{kl}F_{ik}F_{jl}$, then $%
(M,\mathcal{L},g,F,\lambda )$ is called an Einstein Yang-Mills (EYM for
short) system and $\lambda $ is called the Einstein Yang-Mills constant.
Besides Einstein metrics with flat $U(1)$-bundles, other solutions of (\ref%
{eq:1}) are obtained in Section 2 of \cite{St}. The parabolic version of (%
\ref{eq:1}), so called Ricci Yang-Mills flow, is studied for solving (\ref%
{eq:1}) in \cite{St} and \cite{Yo}. A solution of (\ref{eq:1}) also solves
Einstein-Maxwell equations which are studied in the literatures of both
physics and mathematics (cf. \cite{LeB} and references in it). In this
paper, we study the compactness of families of EYM systems.

The convergence of Einstein manifolds in the Gromov-Hausdorff topology has
been studied by many authors (cf. \cite{An1}, \cite{An2}, \cite{BKN}, \cite%
{CT}, \cite{Ti} etc). Gromov's pre-compactness theorem says that, if $%
(M_{i},g_{i})$ is a family of Riemannian $n$-manifolds with diameters
bounded from above and Ricci curvature bounded from below, a subsequence of $%
(M_{i},g_{i})$ converges to a compact length space $(X,d_{X})$ in
the Gromov-Hausdorff sense (cf. \cite{Go}). In addition, if the
Ricci
curvatures $|\mathrm{Ric}(g_{i})|<\mu $, the volumes $\mathrm{Vol}%
_{M_{i}}>v>0$ for constants $\mu $ and $v$ independent of $i$, and the $L^{%
\frac{n}{2}}$-norms of the curvature tensors of $g_{i}$ have an uniform
bound, then it is shown in \cite{An2} that $(X,d_{X})$ is an orbifold with
finite singular points $\{p_{k}\}_{k=1}^{N}$. Furthermore, $d_{X}$ is
induced by a $C^{1,\alpha }$-metric $g_{\infty }$ on the regular part $%
X\backslash \{p_{k}\}_{k=1}^{N}$, and by passing to a subsequence, $g_{i}$ $%
C^{1,\alpha }$-converges to $g_{\infty }$ in the Cheeger-Gromov sense. If $%
g_{i}$ are Einstein metrics with bounded Einstein constants, then $g_{\infty
}$ is an Einstein metric and $g_{i}$ converge to $g_{\infty }$ in the $%
C^{\infty }$-topology (cf. \cite{An1}, \cite{BKN}, \cite{Ti}). In the
present paper, we prove an analogue convergence theorem for EYM systems.

Providing $(M,\mathcal{L},g,F,\lambda )$ is an Einstein Yang-Mills system,
for any $m\in \mathbb{N}^{+}$, $(M,\mathcal{L}^{m},m^{2}g,mF,\frac{\lambda }{%
m^{2}})$ also satisfies (\ref{eq:1}). By choosing an appropriate
$m$, we can normalize the  Einstein Yang-Mills system such that
 the
EYM constant $|\frac{\lambda }{%
m^{2}}|\leq 1$. Thus, we only consider EYM systems with EYM
constants belonging to  $[-1,1]$ in this paper.

\begin{theorem}
\label{thm:1} Let $\left\{ \left( M_{i},\mathcal{L}_{i},g_{i},F_{i},\lambda
_{i}\right) \right\} $ be EYM systems with EYM constants $\lambda _{i}\in %
\left[ -1,1\right] ,$ where $\left\{ M_{i}\right\} $ is a family of
connected closed n-manifolds. Assume that there are constants $\Omega >0$, $%
v>0$, $D>0$, $C_{0}>0$ and $c_{0}>0$ independent of $i$ such that

\begin{itemize}
\item[(i)] $\mathrm{Vol}_{M_{i}}\geq v>0,$ and $\mathrm{diam}_{M_{i}}\leq D,$

\item[(ii)]
\begin{equation*}
\int_{M_{i}}\left\vert F_{i}\right\vert ^{2}d\mu _{i}\leq \Omega ,
\end{equation*}

\item[(iii)] $b_{2}\left( M_{i}\right) \leq c_{0},\text{ for }n=4$, or
\begin{equation*}
\int_{M_{i}}\left\vert \mathrm{Rm}(g_{i})\right\vert ^{\frac{n}{2}}d\mu
_{i}\leq C_{0}\text{ for }n>4.
\end{equation*}
\end{itemize}

Then a subsequence of $(M_{i},g_{i})$ converges, without changing the
subscripts, in the Gromov-Hausdorff sense, to a connected Riemannian
orbifold $(M_{\infty },g_{\infty })$ with finite singular points $%
\{p_{k}\}_{k=1}^{N},$ each having a neighborhood homeomorphic to the cone $%
C\left( S^{n-1}/\Gamma _{k}\right) ,$ with $\Gamma _{k}$ a finite subgroup
of $O\left( n\right) .$ The metric $g_{\infty }$ is a $C^{0}$ Riemannian
orbifold metric on $M_{\infty },$ which is smooth off the singular points.
Furthermore, there is a $U(1)$-bundle $\mathcal{L}_{\infty }$ on the regular
part $M_{\infty }^{o}=M_{\infty }\backslash \{p_{k}\}_{k=1}^{N}$, a
Yang-Mills connection $A_{\infty }$ of $\mathcal{L}_{\infty }$ with
curvature $F_{\infty }$, and a constant $\lambda _{\infty }\in [-1,1]$ such
that $(M_{\infty }^{o},\mathcal{L}_{\infty },g_{\infty },F_{\infty },\lambda
_{\infty })$ is an EYM system. And, for any compact subset $K\subset \subset
M_{\infty }^{o}$, there are embeddings $\Phi _{K}^{i}:K\rightarrow M_{i}$
such that $\Phi _{K}^{i,-1}\mathcal{L}_{i}\cong \left. \mathcal{L}_{\infty
}\right\vert _{K}$ for $i\gg 1$,
\begin{equation*}
\Phi _{K}^{i,\ast }g_{i}\rightarrow g_{\infty },\ \ \ \Phi _{K}^{i,\ast
}F_{i}\rightarrow F_{\infty },\ \ \ \text{and}\ \ \ \lambda _{i}\rightarrow
\lambda _{\infty },
\end{equation*}%
when $i\rightarrow \infty $ in the $C^{\infty }$-sense.
\end{theorem}

\begin{re}
If we assume that $\lambda _{i }> \kappa >0$ for a uniform positive
constant $\kappa$, then the Ricci curvature of $g_{i}$ is bounded
below by $\kappa$, since $\eta$ is a non-negative symmetric tensor.
By Myers' Theorem, the diameters of $g_{i}$ are uniformly bounded
from above. Therefore, the condition of diameter bounds in Theorem
\ref{thm:1} can be removed.
\end{re}

\begin{re}
If $M_{i}$ in Theorem \ref{thm:1} are odd-dimensional oriented
manifolds, by Corollary 2.8 in \cite{An2}, there is no singular
points in $M_{\infty }$,
and the EYM systems $\left\{ \left( M_{i},\mathcal{L}_{i},g_{i},F_{i},%
\lambda _{i}\right) \right\} $ smoothly converges to a smooth EYM system $%
(M_{\infty },\mathcal{L}_{\infty },g_{\infty },F_{\infty },\lambda
_{\infty })$.
\end{re}

\begin{re}
If we replace the assumption $\mathrm{Vol}_{M_{i}}\geq v>0$ and $%
\int_{M_{i}}\left\vert \mathrm{Rm}\right\vert ^{\frac{n}{2}}d\mu
\leq C_{0}$ by
the condition of injective radius bounded from below, we can  obtain $%
C^{\infty }$ subconvergence of the sequence $\left\{ \left( M_{i},\mathcal{L}%
_{i},g_{i},F_{i},\lambda _{i}\right) \right\} $ to a smooth EYM system $%
(M_{\infty },\mathcal{L}_{\infty },g_{\infty },F_{\infty },\lambda _{\infty
})$, by Theorem 1.1 in \cite{An2} and the similar arguments in the proof of
Theorem \ref{thm:1}.
\end{re}

From Theorem 1, one can see that  orbifolds with orbifold metrics
solving Einstein Yang-Mills equations on regular parts appear
naturally as limits of sequences of EYM systems.  We would like to
construct such orbifolds which do not admit any Einstein orbifold
metrics. Firstly, let's recall an example of EYM systems from
\cite{St}. Let $g_{1}$ be the standard metric with Gaussian
curvature $1$ on $S^{2}$, $g_{2}$ be the standard metric with
Gaussian curvature $-1$ on a surface $H$ with genus $\mathfrak{g}$
bigger than 1, and $\omega _{1}$ (resp. $\omega _{2}$) be the volume
form of $g_{1}$ (resp. $g _{2}$). If $US^{2}$ and $UH$ denote the
unit tangent bundles of $S^{2}$ and $H$ respectively, then $F=\pi
_{1}^{\ast }\omega _{1}+\pi _{2}^{\ast }\omega _{2}$ is the
curvature of $\mathcal{L}=\pi _{1}^{-1}US^{2}\otimes \pi
_{2}^{-1}UH$ on $M=S^{2}\times H$, where $\pi _{1} $ and $\pi _{2}$
are standard projections from $M$ to $S^{2}$ and $H$. For any
$\lambda <0$, the Riemannian metric $g=A_{\lambda }\pi _{1}^{\ast
}g_{1}+B_{\lambda }\pi _{2}^{\ast }g_{2}$ and $F$ solve the Einstein
Yang-Mills equations (\ref{eq:1}), i.e. $(M,\mathcal{L},g,F,\lambda
)$ is an EYM system, where $A_{\lambda }=\frac{1-\sqrt{1-2\lambda
}}{2\lambda }$ and $B_{\lambda }=\frac{-1-\sqrt{1-2\lambda
}}{2\lambda }$. Now we assume that $H$ is a
hyperelliptic Riemann surface, i.e. $H$ admits a conformal involution with $2%
\mathfrak{g}+2$ fixed points. Consider the involution $\iota $ of
$M$  obtained as the product of a $180^{\circ }$ rotation of $S^{2}$
around an
axis and the hyperelliptic involution of $H$. It is clear that the $\mathbb{Z%
}_{2}$-action induced by the involution preserves $\mathcal{L}$, $g$ and $F$%
, i.e. $\iota ^{-1}\mathcal{L}\cong \mathcal{L}$, $\iota ^{\ast }g=g$ and $%
\iota ^{\ast }F=F$. Thus the orbifold $M/\mathbb{Z}_{2}$ has $4\mathfrak{g}%
+4 $ singular points, $g$ induces an orbifold Riemannian metric on $M/%
\mathbb{Z}_{2}$, and $(\mathcal{L},g,F,\lambda )$ induces an EYM system on
the regular part of $M/\mathbb{Z}_{2}$. We claim that $M/\mathbb{Z}_{2}$
wouldn't admit any orbifold Einstein metric. Note that the orbifold Euler
characteristic $\chi_{orb} (M/\mathbb{Z}_{2})=\frac{1}{2}\chi (M)=2-2%
\mathfrak{g}<0$. However, if there is an orbifold Einstein metric $g^{\prime
}$ on $M/\mathbb{Z}_{2}$, then
\begin{equation*}
\chi_{orb} (M/\mathbb{Z}_{2})=\frac{1}{8\pi ^{2}}\int_{M/\mathbb{Z}_{2}}|%
\mathrm{Rm}(g^{\prime })|^{2}d\mu \geq 0,
\end{equation*}%
by (6.2) in \cite{An1}, which is a contradiction.

In Section $2,$ some properties of EYM systems are studied, and, in Section $%
3$ Theorem 1 is proved.

\section{Preliminary properties of EYM systems}

To start things off we present some basic preliminary properties of
EYM systems which are  used in the proof of  Theorem 1. Firstly, we
have the following lemma for an EYM system.

\begin{lm}
\label{lam:1} If $(M^{n},\mathcal{L},g,F,\lambda )$ is an EYM system, then
the scalar curvature $R$ of $g$ and $\left\vert F\right\vert ^{2}$ are
constants.
\end{lm}

\begin{proof}
By taking trace of the first equation of (\ref{eq:1}), we have%
\begin{equation*}
R-\frac{1}{2}\left\vert F\right\vert ^{2}=n\lambda .
\end{equation*}%
Hence
\begin{eqnarray*}
0 &=&\nabla _{k}\left( R-\frac{1}{2}\left\vert F\right\vert ^{2}\right) \\
&=&2g^{ij}\nabla _{i}R_{jk}-\frac{1}{2}\nabla _{k}\left\vert F\right\vert
^{2} \\
&=&2g^{ij}\nabla _{i}\left( \lambda g_{jk}+\frac{1}{2}\eta _{jk}\right) -%
\frac{1}{2}\nabla _{k}\left\vert F\right\vert ^{2} \\
&=&g^{ij}\nabla _{i}\eta _{jk}-\frac{1}{2}\nabla _{k}\left\vert
F\right\vert ^{2}\text{ }.
\end{eqnarray*}%
And on the other hand,%
\begin{eqnarray*}
g^{ij}\nabla _{i}\eta _{jk} &=&g^{ij}\nabla _{i}\left(
g^{pq}F_{jp}F_{kq}\right) \\
&=&g^{ij}g^{pq}F_{jp}\nabla _{i}F_{kq}+g^{ij}g^{pq}F_{kq}\nabla _{i}F_{jp} \\
&=&g^{ij}g^{pq}F_{jp}\nabla _{i}F_{kq}-g^{pq}F_{kq}d^{\ast }F_{p} \\
&=&-g^{ij}g^{pq}F_{jp}\left( \nabla _{k}F_{qi}+\nabla _{q}F_{ik}\right) \\
&=&\frac{1}{4}\nabla _{k}\left\vert F\right\vert ^{2}\text{ },
\end{eqnarray*}%
where we have used the Yang-Mills equation of the EYM system (\ref{eq:1})
and the Bianchi identity. Thus we have $\nabla _{k}\left\vert F\right\vert
^{2}=0,$ and then $\nabla _{k}R=0$. We obtain the conclusion, i.e. $R$ and $%
\left\vert F\right\vert ^{2}$ are constants.
\end{proof}

As a consequence of Lemma \ref{lam:1}, the  $W^{1,2p}$ esitmate of
the bundle curvature $F$ of an EYM system is easily obtained.

\begin{prop}
\label{prop:1} Let $(M,\mathcal{L},g,F,\lambda )$ be an EYM system
with $|\lambda|\leq 1$. If the volume of the underlying $n$-manifold
$M$ is bounded from above by $V$, then for any $1<$ $p<+\infty ,$
\begin{equation*}
\left( \int_{M}\left\vert \nabla F\right\vert ^{2p}d\mu \right) ^{\frac{1}{p}%
}\leq 2\max \{1,2V^{\frac{1}{p}}\}\left\vert F\right\vert ^{2}
\left( 1+\left( \int_{M}\left\vert \mathrm{Rm}\right\vert ^{p}d\mu
\right) ^{\frac{1}{p}}\right) .
\end{equation*}%
 Moreover, there exists a
constant $C=C(p, V,\left\vert F\right\vert ^{2})>0$ such that
$$\left\| F\right\| _{W^{1,2p}}^{2}\leq C\left( 1+\left\|
\mathrm{Rm}\right\|_{L^{p}}\right) .$$
\end{prop}

\begin{proof}
We know that $\left\vert F\right\vert ^{2}$ is a constant due to Lemma \ref%
{lam:1}. Hence
\begin{eqnarray*}
0 &=&\Delta \left\vert F\right\vert ^{2}=2\left\langle \Delta
F,F\right\rangle +2\left\vert \nabla F\right\vert ^{2} \\
&=&2\left\langle \Delta _{d}F,F\right\rangle -2\left(
2R_{ijkl}F_{jk}F_{il}-R_{ik}F_{kj}F_{ij}-R_{jk}F_{ik}F_{ij}\right)
+2\left\vert \nabla F\right\vert ^{2} \\
&=&-4R_{ijkl}F_{jk}F_{il}+4\left\langle \mathrm{Ric},\eta \right\rangle
+2\left\vert \nabla F\right\vert ^{2}.
\end{eqnarray*}%
It follows that%
\begin{eqnarray*}
\left\vert \nabla F\right\vert ^{2} &=&2\left\vert F\right\vert
^{2}\left\vert \mathrm{Rm}\right\vert -2\left\langle \mathrm{Ric},\eta
\right\rangle \\
&=&2\left\vert F\right\vert ^{2}\left\vert \mathrm{Rm}\right\vert
-2\left\langle \frac{1}{2}\eta +\lambda g,\eta \right\rangle \\
&\leq &2\left\vert F\right\vert ^{2}\left\vert \mathrm{Rm}\right\vert
-2\lambda \left\vert F\right\vert ^{2} \\
&\leq &2 \left\vert F\right\vert ^{2}\left( 1+\left\vert
\mathrm{Rm}\right\vert \right) ,
\end{eqnarray*}%
by $|\lambda|\leq 1.$  \ So%
\begin{eqnarray*}
\left( \int_{M}\left\vert \nabla F\right\vert ^{2p}d\mu \right) ^{\frac{1}{p}%
} &\leq &\left( \int_{M}\left( 2\left\vert F\right\vert ^{2}\left(
1+\left\vert \mathrm{Rm}\right\vert \right) \right) ^{p}d\mu \right) ^{\frac{%
1}{p}} \\
&\leq &2\left\vert F\right\vert ^{2}\left[ \left( \int_{\left\vert
\mathrm{Rm}\right\vert \leq 1}2^{p}d\mu \right)
^{\frac{1}{p}}+\left( \int_{\left\vert \mathrm{Rm}\right\vert
>1}\left\vert \mathrm{Rm}\right\vert
^{p}d\mu \right) ^{\frac{1}{p}}\right] \\
&\leq &2\left\vert F\right\vert ^{2}\left[ \left( \int_{M}2^{p}d\mu
\right) ^{\frac{1}{p}}+\left( \int_{M}\left\vert
\mathrm{Rm}\right\vert
^{p}d\mu \right) ^{\frac{1}{p}}\right] \\
&\leq &2\max \{1,2V^{\frac{1}{p}}\}\left\vert F\right\vert
^{2}\left(
1+\left( \int_{M}\left\vert \mathrm{Rm}\right\vert ^{p}d\mu \right) ^{\frac{1%
}{p}}\right) .
\end{eqnarray*}%
Combining with the fact that $\left\vert F\right\vert ^{2}$ is a
constant, we have $$\left\| F\right\| _{W^{1,2p}}^{2}\leq C\left(
1+\left\| \mathrm{Rm}\right\|_{L^{p}}\right) .$$
\end{proof}

After the classical Cheeger-Gromov's theorem, many works have been
done in the convergence of families of manifolds, and there are also
many applications of these convergence results in solving geometric
problems (cf. \cite{An3}, \cite{Go} and \cite{GW} etc.). A key point
 in  the proofs of these convergence results is to use the harmonic
coordinates, that is, the corresponding coordinate functions are
harmonic functions. One reason to use harmonic coordinates is that
Ricci tensors are elliptic operators of metric tensors under such
coordinates.

Given tensors $\xi $ and $\zeta ,$ $\xi \ast \zeta $ denotes some linear
combination of contractions of $\xi \otimes \zeta $ in this paper.

\begin{lm}
\label{lam:2} Let $(M,\mathcal{L},g,F,\lambda )$ be an EYM system and $%
u\colon U\rightarrow R^{n}$ be a harmonic coordinate system of the
underlying manifold $M.$ Then in this coordinate, the EYM equations (\ref%
{eq:1}) are %
\begin{equation}
-\frac{1}{2}g^{kl}\frac{\partial ^{2}g_{ij}}{\partial u^{k}\partial u^{l}}%
-Q_{ij}\left( g,\partial g\right)
-\frac{1}{2}g^{kl}F_{ik}F_{jl}-\lambda g_{ij}=0,  \tag{3}
\label{eq:3}
\end{equation}%
\begin{equation}
g^{kl}\frac{\partial ^{2}F_{ij}}{\partial u^{k}\partial u^{l}}%
+P_{ij}(g,\partial g,\partial F)+T_{ij}(g,\partial g,F)=0,  \tag{4}
\label{eq:4}
\end{equation}%
where
\begin{equation*}
Q\left( g,\partial g\right) =\left( g^{-1}\right) ^{\ast 2}\ast \left(
\partial g\right) ^{\ast 2},
\end{equation*}%
\begin{equation*}
P(g,\partial g,\partial F)=\left( g^{-1}\right) ^{\ast 2}\ast \partial g\ast
\partial F,
\end{equation*}%
and%
\begin{equation*}
T(g,\partial g,\partial ^{2}g,F)=\left( g^{-1}\right) ^{\ast 3}\ast \left(
\partial g\right) ^{\ast 2}\ast F+\left( g^{-2}\right) ^{\ast 2}\ast
\partial ^{2}g\ast F.
\end{equation*}
\end{lm}

\begin{proof}
The Ricci tensor under the harmonic coordinate system is given by (cf. \cite%
{Pe})%
\begin{equation*}
R_{ij}=-\frac{1}{2}g^{kl}\frac{\partial ^{2}g_{ij}}{\partial u^{k}\partial
u^{l}}-Q_{ij}\left( g,\partial g\right)
\end{equation*}%
with
\begin{equation*}
Q_{ij}\left( g,\partial g\right) =\left( g^{-1}\right) ^{\ast 2}\ast \left(
\partial g\right) ^{\ast 2}.
\end{equation*}%
The first equation of EYM equations becomes%
\begin{equation*}
-\frac{1}{2}g^{kl}\frac{\partial ^{2}g_{ij}}{\partial u^{k}\partial u^{l}}%
-Q_{ij}\left( g,\partial g\right) -\frac{1}{2}g^{kl}F_{ik}F_{jl}-\lambda
g_{ij}=0.
\end{equation*}%
Note that the Yang-Mills equation $d^{\ast }F=0$ is equivalent to\text{ }$%
\Delta _{d}F=0$ on compact manifolds, i.e. $F$ is harmonic. We use the
Bochner formula of $2$-forms,%
\begin{equation*}
\left( \Delta _{d}F\right) _{ij}=\left( \Delta F\right)
_{ij}+2g^{jp}g^{kq}R_{ijkl}F_{pq}-g^{kl}R_{ik}F_{lj}-g^{kl}R_{jk}F_{il}.
\end{equation*}%
 Observe that $%
\Delta u^{k}=0,$ which is equivalent to $g^{ij}\Gamma _{ij}^{k}=0$ for all $%
k=1,\cdots ,n.$ Then  in this coordinate chart,
\begin{eqnarray*}
\left( \Delta _{d}F\right) _{ij} &=&g^{kl}\frac{\partial ^{2}F_{ij}}{%
\partial u^{k}\partial u^{l}}+g^{kp}g^{lq}R_{ijkl}F_{pq} \\
&&-2g^{pq}\{\Gamma _{pi}^{l}\frac{\partial F_{lj}}{\partial u^{q}}+\Gamma
_{pj}^{l}\frac{\partial F_{il}}{\partial u^{q}}\}+\frac{\partial g^{pq}}{%
\partial u^{i}}\Gamma _{pq}^{l}F_{lj}+\frac{\partial g^{pq}}{\partial u^{j}}%
\Gamma _{pq}^{l}F_{il} \\
&&+2g^{pq}\{\Gamma _{iq}^{m}\Gamma _{mp}^{l}F_{lj}+\Gamma _{jq}^{m}\Gamma
_{mp}^{l}F_{il}+\Gamma _{jq}^{m}\Gamma _{ip}^{l}F_{lm}\}.
\end{eqnarray*}%
And the Yang-Mills equation becomes%
\begin{equation*}
g^{kl}\frac{\partial ^{2}F_{ij}}{\partial u^{k}\partial u^{l}}%
+P_{ij}(g,\partial g,\partial F)+T_{ij}(g,\partial g,\partial ^{2}g,F)=0,
\end{equation*}%
where
\begin{equation*}
P_{ij}(g,\partial g,\partial F)=\left( g^{-1}\right) ^{\ast 2}\ast \partial
g\ast \partial F, \ \ \mathrm{and}
\end{equation*}
\begin{equation*}
\ \ T_{ij}(g,\partial g,\partial ^{2}g,F)=\left( g^{-1}\right) ^{\ast 3}\ast
\left( \partial g\right) ^{\ast 2}\ast F+\left( g^{-2}\right) ^{\ast 2}\ast
\partial ^{2}g\ast F.
\end{equation*}
\end{proof}

\section{Convergence of EYM metrics}

This section is devoted to the proof of Theorem 1. We begin with a review of
 harmonic coordinates which play essential roles in the
Cheeger-Gromov convergence of Riemannian manifolds.

A compact Riemannian $n$-manifold $(M,g)$ is said to have an
$(r,\sigma ,C^{l,\alpha })$ for $0<\alpha <1$ (resp. $(r,\sigma
,W^{l,p })$ for $1<p<\infty$) adapted harmonic atlas, if there is a
covering $\left\{
B_{x_{k}}\left( r\right) \right\} _{k=1}^{\sigma }$ of $M$ by geodesic $r$%
-balls, for which the balls $B_{x_{k}}\left( \frac{r}{2}\right) $ also cover
$M$ and the balls $B_{x_{k}}\left( \frac{r}{4}\right) $ are disjoint, such
that each $B_{x_{k}}\left( 10r\right) $ has a harmonic coordinate chart $%
u\colon B_{x_{k}}\left( 10r\right) \rightarrow U\subset
\mathbb{R}
^{n},$ and the metric tensor in these coordinates are $C^{l,\alpha
}$  (resp. $W^{l,p}$)
bounded, i.e. if $g_{ij}=g\left( \frac{\partial }{\partial u_{i}},\frac{%
\partial }{\partial u_{j}}\right) $ on $B_{x_{k}}\left( 10r\right) ,$ then%
\begin{equation*}
C^{-1}\delta _{ij}\leq g_{ij}\leq C\delta _{ij},
\end{equation*}%
and%
\begin{equation*}
\left\Vert g_{ij}\right\Vert _{C^{l,\alpha }}\leq C, \ \ \ \ ({\rm
resp. } \ \ \  \left\Vert g_{ij}\right\Vert _{W^{l,p}}\leq C,)
\end{equation*}%
for some constant $C>1,$ where the norms are taken with respect to the
coordinates $\left\{ u^{j}\right\} $ on $B_{x_{k}}\left( 10r\right) .$

The following is a version of the local Cheeger-Gromov compactness theorem
(cf. Theorem 2.2 in \cite{An1}, Lemma 2.1 in \cite{An2} and \cite{GW}).

\begin{prop}
\label{prop:2}\bigskip Let $V_{i}$ be a sequence of domains in closed $%
C^{\infty }$ Riemannian manifolds $(M_{i},g_{i})$ such that $V_{i}$
admits an adapted harmonic atlas $\left( \delta ,\sigma ,C^{l,\alpha
}\right) $ for a constant $C>1$. Then there is a subsequence which
converges uniformly on compact subsets in
the $C^{l,\alpha^{\prime }}$ topology, $\alpha^{\prime }<\alpha$, to a $%
C^{l,\alpha }$ Riemannian manifold $V_{\infty }.$
\end{prop}

In \cite{An2}, Anderson shows a local existence of an adapted atlas.

\begin{prop}
\label{prop:3}\textrm{(Proposition 2.5, Lemma 2.2 and Remarks 2.3 in \cite%
{An2})} Let $(M,g)$ be a Riemannian manifold with
\begin{equation*}
\left\vert \mathrm{Ric}_{M}\right\vert \leq \Lambda ,\ \ \mathrm{diam}%
_{M}\leq D\ \ and\ \ \mathrm{Vol}_{B\left( r\right) }\geq v_{0}>0
\end{equation*}%
for a ball $B\left( r\right) $.  Then, for any $C>1$, there exist
positive
constants $ \sigma=\sigma \left( \Lambda ,v_{0},n,D \right)$, $\epsilon =\epsilon \left( \Lambda ,v_{0},n,C\right) $ and $%
\delta =\delta \left( \Lambda ,v_{0},n,C\right) ,$ such that for any $%
1<p<\infty ,$ one can obtain a $\left( \delta ,\sigma ,W^{2,p}\right) $
adapted atlas on the union $U$ of those balls $B\left( r\right) $ satisfying
\begin{equation*}
\int_{B\left( 4r\right) }\left\vert \mathrm{Rm}\right\vert ^{\frac{n}{2}%
}d\mu \leq \epsilon .
\end{equation*}%
More precisely, on any $B\left( 10\delta \right) \subset U$, there is a
harmonic coordinate chart such that for any $1<p<\infty ,$%
\begin{equation*}
C^{-1}\delta _{ij}\leq g_{ij}\leq C\delta _{ij},
\end{equation*}%
and%
\begin{equation*}
\left\Vert g_{ij}\right\Vert _{W^{2,p}}\leq C.
\end{equation*}
\end{prop}

Now we need the following lemma before proving Theorem \ref{thm:1}.

\begin{lm}
\label{lem:3} Let $(M,\mathcal{L},g,F,\lambda )$ be an EYM system with $%
\lambda \in \left[ -1,1\right] .$ Assume that
\begin{equation*}
\mathrm{Vol}_{M}\geq v>0, \ \ and \ \ \int_{M}\left\vert F\right\vert
^{2}d\mu \leq \Omega ,
\end{equation*}%
then there is a constant $C_{1}=C_{1}(v,\Omega ,n)$ such that
\begin{equation*}
\left\vert \mathrm{Ric}\left( g\right) \right\vert \leq C_{1}.
\end{equation*}
Furthermore, if $M$ is a $4$-manifold and
\begin{equation*}
\mathrm{diam}_{M}\leq D, \ \ \ b_{2}\left( M\right) \leq c_{0},
\end{equation*}
then there is a constant $C_{2}=C_{2}(v,\Omega ,D,c_{0})$ such that%
\begin{equation*}
\int_{M}\left\vert \mathrm{Rm}\left( g\right) \right\vert ^{2}d\mu \leq
C_{2}.
\end{equation*}
\end{lm}

\begin{proof}
Since $(M,\mathcal{L},g,F,\lambda )$ is an EYM system, by Lemma \ref{lam:1},
$\left\vert F\right\vert $ is a constant, and
\begin{equation*}
\left\vert F\right\vert ^{2}\mathrm{Vol}_{M}=\int_{M}\left\vert F\right\vert
^{2}d\mu \leq \Omega .
\end{equation*}%
By the assumption $\mathrm{Vol}_{M}\geq v>0,$ we have
\begin{equation*}
\left\vert F\right\vert ^{2}\leq \frac{\Omega }{v}.
\end{equation*}%
Then it is clear that
\begin{eqnarray*}
\left\vert \mathrm{Ric}\left( g\right) \right\vert ^{2} &=&\left\vert \frac{1%
}{2}\eta \left( g,F\right) +\lambda g\right\vert ^{2} \\
&\leq &\frac{1}{4}\left\vert F\right\vert ^{4}+\lambda \left\vert
F\right\vert ^{2}+n\lambda ^{2} \\
&\leq &C_{1}(\Omega ,v,n).
\end{eqnarray*}

From now on, we assume that the dimension of $M$ is $4$, $\mathrm{diam}%
_{M}\leq D,$ and $b_{2}\left( M\right) \leq c_{0}.$ The Bishop
comparison theorem implies  that the volume of $M$ is bounded above
by a constant $c\left( \Omega ,v,D,n\right) ,$ i.e.
$$\mathrm{Vol}_{M}\leq c\left( \Omega ,v,D,n\right). $$

By the standard Gauss-Bonnet-Chern formula (cf. \cite{Be}), the Euler
characteristic $\chi \left( M\right) $ of an oriented compact $4$-manifold $%
M $ can be expressed as follows,%
\begin{equation*}
\chi \left( M\right) =\frac{1}{8\pi ^{2}}\int_{M}\left( \left\vert \mathrm{Rm%
}\right\vert ^{2}-4\left\vert \mathrm{Ric}\right\vert ^{2}+R^{2}\right) d\mu
.
\end{equation*}%
Thus%
\begin{eqnarray*}
\frac{1}{8\pi ^{2}}\int_{M}\left\vert \mathrm{Rm}\left( g\right) \right\vert
^{2}d\mu &=&\chi \left( M\right) +\frac{1}{8\pi ^{2}}\int_{M}\left(
4\left\vert \mathrm{Ric}\left( g\right) \right\vert ^{2}-R^{2}\left(
g\right) \right) d\mu \\
&\leq &2+b_{2}\left( M\right) +\frac{\max_{M}\left\vert \mathrm{Ric}\left(
g\right) \right\vert ^{2}}{2\pi ^{2}}\mathrm{Vol}_{M} \\
&\leq &C_{2}(v,\Omega ,D,c_{0}).
\end{eqnarray*}
\end{proof}

Now we are ready to prove the main theorem of this paper.

\begin{proof}[Proof of Theorem \protect\ref{thm:1}]
By Lemma \ref{lem:3} we know that under the assumption of Theorem 1
the Ricci curvatures and  the $L^{\frac{n}{2}}$-norms of the
curvature tensors of $g_{i}$ have uniform bounds. By Theorem 2.6 in
\cite{An2}, a
subsequence of $(M_{i},g_{i})$ converges to a  Riemannian orbifold $%
(M_{\infty},g_{\infty})$ with finite isolated singular points in the
Gromov-Hausdorff sense. Furthermore, $g_{i}$ $C^{1,\alpha}$-converges to $%
g_{\infty}$ on the regular part in the Cheeger-Gromov sense. Now we improve
the convergence to the $C^{\infty}$-sense, and study the convergence of $%
F_{i}$.

For a given $r>0,$ let $\{B_{x_{k}}^{i}(r)\}$ be a family of metric balls of
radius $r$ such that $\{B_{x_{k}}^{i}(r)\}$ covers $(M_{i},g_{i})$, and   $B_{x_{k}}^{i}(\frac{r}{2%
})$ are disjoint. Denote%
\begin{equation*}
G_{i}(r)=\cup \left\{ \left. B_{x_{k}}^{i}(r)\right\vert
\int_{B_{x_{k}}^{i}(4r)}\left\vert \mathrm{Rm}\left( g_{i}\right)
\right\vert ^{2}d\mu _{i}\leq \epsilon \right\} ,
\end{equation*}%
where $\epsilon =\epsilon \left( \Omega ,v,D,n,C_{3}\right) >0$ is
obtained in Proposition \ref{prop:3} for a constant $C_{3}>1$. So
$G_{i}(r)$ are covered by a  $\left( \delta ,\sigma ,W^{2,p}\right)
$ $\left( \text{for any }1<p<\infty \right) $ adapted atlas with the
harmonic radius uniformly bounded from below by virtue of
Proposition \ref{prop:3}. In these coordinates we have $W^{2,p}$
bounds for the metrics, i.e.
\begin{equation*}
C_{3}^{-1}\delta _{jk}\leq g_{i,jk}\leq C_{3}\delta _{jk},
\end{equation*}%
and%
\begin{equation*}
\left\Vert g_{i}\right\Vert _{W^{2,p}}\leq C_{3}.
\end{equation*}%
Since the curvature tensor $\mathrm{Rm}\left( g\right) =\partial
^{2}g+\partial g\ast \partial g,$ for any $1<p<\infty ,$ we have%
\begin{equation*}
\left\Vert \mathrm{Rm}(g_{i})\right\Vert _{L^{p}}\leq C\left( n,\left\Vert
g\right\Vert _{W^{2,2p}}\right) .
\end{equation*}%
By Proposition $1$ and Lemma 3, we get an uniform $W^{1,2p}$ estimate for
the $2$-forms $F_{i},$%
\begin{eqnarray*}
\left\Vert F_{i}\right\Vert _{W^{1,2p}}^{2} &\leq &C\left( 1+\left\Vert
\mathrm{Rm}(g_{i})\right\Vert _{L^{p}}\right)  \\
&\leq &C(p,n,\Omega ,v,C_{3}).
\end{eqnarray*}%
By the Sobolev embbeding theorem we know that\ for all $0<\alpha \leq 1-%
\frac{n}{2p},$%
\begin{equation*}
\left\Vert g_{i}\right\Vert _{C^{1,\alpha }}\leq C_{4},\text{ \ }
\end{equation*}%
\begin{equation*}
\left\Vert F_{i}\right\Vert _{C^{\alpha }}^{2}\leq C_{5},
\end{equation*}%
 where
 $C_{4}$ and $C_{5}>0$ are two constants.

Note that the EYM equations are elliptic equations (\ref{eq:3}) and (\ref%
{eq:4}) in harmonic coordinates by Lemma 2. Now we apply the
Schauder
estimate for elliptic partial differential equations (cf. \cite{GT}) to (\ref%
{eq:3}), and obtain
\begin{equation*}
\left\Vert g_{i}\right\Vert _{C^{2,\alpha }}\leq C(n,\left\Vert
g_{i}\right\Vert _{C^{1,\alpha }},\left\Vert F_{i}\right\Vert _{C^{\alpha
}}^{2})\leq C_{6},
\end{equation*}%
for a constant $C_{6}>0$ independent of $i$. By applying $L^{p}$-estimates
for elliptic differential equations (cf. \cite{GT}) to (\ref{eq:4}),
\begin{equation*}
\left\Vert F_{i}\right\Vert _{W^{2,2p}}\leq C(n,\left\Vert
g_{i}\right\Vert _{C^{2,\alpha }},\left\Vert F_{i}\right\Vert
_{L^{2p}}^{2})\leq C_{7},
\end{equation*}%
where $C_{7}>0$ is a constant independent of $i$. Then
\begin{equation*}
\left\Vert F_{i}\right\Vert _{C^{1,\alpha }}\leq C_{8},
\end{equation*}%
by the Sobolev embedding theorem. By standard elliptic theory (cf. \cite{GT}%
), for any $l\geq 0$ and $0<\alpha <1,$ we obtain uniform $C^{l,\alpha }$
bounds for the metrics $g_{i}$ and the bundle curvatures $F_{i},$%
\begin{equation*}
\left\Vert g_{i}\right\Vert _{C^{l,\alpha }}\leq C(l,\alpha ,n,\Omega
,v,C_{3}),
\end{equation*}%
\begin{equation*}
\left\Vert F_{i}\right\Vert _{C^{l-1,\alpha }}\leq C(l,\alpha ,n,\Omega
,v,C_{3}).
\end{equation*}%
Hence all the covariant derivatives of the curvature tensor have uniform
bounds. By Proposition \ref{prop:2}, there is a subsequence of $G_{i}\left(
r\right) $ converges in the $C^{\infty }$ topology, to an open manifold $%
G_{r}$ with a smooth metric $g_{r}$ and a harmonic $2$-form $F_{r}.$

Because $\int_{M_{i}}\left\vert \mathrm{Rm}\left( g_{i}\right) \right\vert ^{%
\frac{n}{2}}d\mu _{i}$ is uniformly bounded, there are finitely many of
disjoint balls $B\left( r\right) $ on which $\int_{B\left( 4r\right)
}\left\vert \mathrm{Rm}\right\vert ^{2}d\mu >\epsilon .$ And the number of
such disjoint balls are independent of $r.$

Now choose $r_{j}\rightarrow 0,$ with $r_{j+1}\leq
\frac{1}{2}r_{j}.$ Let
\begin{equation*}
G_{i}^{j}=\left\{ \left. x\in M\right\vert x\in G_{i}\left( r_{m}\right)
\text{ for some }m\leq j\right\} ,
\end{equation*}%
then%
\begin{equation*}
G_{i}^{1}\subset G_{i}^{2}\subset G_{i}^{3}\subset \cdots \subset M_{i}.
\end{equation*}%
By using the diagonal argument,  a subsequence of $\left(
G_{i}^{j},g_{i},F_{i},\lambda _{i}\right) ,$  converges, in the $%
C^{l,\alpha }$ topology for all $l\geq 0$, to an open manifold $M_{\infty
}^{o}$ with a smooth Riemannian metric $\widehat{g}_{\infty }$, a harmonic $%
2 $-form $F_{\infty }$ and a constant $\lambda _{\infty }\in \left[ -1,1%
\right] ,$ such that $\left( \widehat{g}_{\infty },F_{\infty },\lambda
_{\infty }\right) $ satisfies the EYM equations. More precisely, for any
compact subset $K\subset \subset M_{\infty }^{o},$ there are embeddings $%
\Phi _{K}^{i}:K\rightarrow M_{i}$ such that
\begin{equation}  \label{eq:5}
\Phi _{K}^{i,\ast }g_{i}\overset{C^{l,\alpha }}{\rightarrow }\widehat{g}%
_{\infty },\ \ \ \Phi _{K}^{i,\ast }F_{i}\overset{C^{l,\alpha }}{\rightarrow
}F_{\infty },\ \ \ \text{and}\ \ \ \lambda _{i}\rightarrow \lambda _{\infty
},  \tag{5}
\end{equation}%
when $i\rightarrow \infty $ for any $l\geq 0$ and $0<\alpha <1$.

For a compact subset $K\subset \subset M_{\infty }^{o},$ we fix a
decomposition $H^{2}(K; \mathbb{Z})\cong H^{2}_{F}(K; \mathbb{Z})\oplus
H^{2}_{T}(K; \mathbb{Z})$ where $H^{2}_{T}(K; \mathbb{Z})$ is the torsion
part and $H^{2}_{F}(K; \mathbb{Z})$ is the free part,  which is a lattice in $%
H^{2}(K; \mathbb{R})$. Then the first Chern class $c_{1}(\Phi _{K}^{i,-1}%
\mathcal{L}_{i})=c_{1,i}^{F}+c_{1,i}^{T}$, where $c_{1,i}^{F}\in
H^{2}_{F}(K; \mathbb{Z})$, $c_{1,i}^{T}\in H^{2}_{T}(K; \mathbb{Z})$, and $%
\frac{\sqrt{-1}}{2\pi }\Phi _{K}^{i,*}F_{i}$ represents $c_{1,i}^{F}$ by the
standard Chern-Weil theory, i.e. $\left[ \frac{\sqrt{-1}}{2\pi }\Phi
_{K}^{i,*}F_{i}\right]=c_{1,i}^{F}\in H^{2}_{F}(K; \mathbb{Z})\subset
H^{2}_{F}(K; \mathbb{R}) $. By (\ref{eq:5}), $c_{1,i}^{F}=\left[ \frac{\sqrt{%
-1}}{2\pi }\Phi _{K}^{i,*} F_{i}\right]$ converges to $c_{1,\infty}^{F}=%
\left[ \frac{\sqrt{-1}}{2\pi }F_{\infty}\right]$ in $H^{2}_{F}(K; \mathbb{R}%
) $, which implies that, by passing to a subsequence, $c_{1,i}^{F}
=c_{1,\infty}^{F}$  for $i\gg 1$. Since there are only finite
elements in $H^{2}_{T}(K; \mathbb{Z})$, by passing to a subsequence,
we obtain that $c_{1,i}^{T}\equiv c_{1,\infty}^{T}$
in $H^{2}_{T}(K; \mathbb{Z})$. Hence there is a $U(1)$-bundle $\mathcal{L}%
_{K,\infty}$ on $K$ such that the first Chern class $c_{1}(\mathcal{L}%
_{K,\infty})=c_{1,\infty}^{F}+c_{1,\infty}^{T}$, $\mathcal{L}%
_{K,\infty}=\Phi _{K}^{i,-1}\mathcal{L}_{i}$ and $F_{\infty}$ is the
curvature of a Yang-Mills connection on $\mathcal{L}_{K,\infty}$. Now we
take a sequence of compact subsets $K_{1}\subset K_{2}\subset \cdots \subset
M_{\infty }^{o}$ with $\cup K_{a}=M_{\infty }^{o}$. By the standard diagonal
argument, there is a $U(1)$-bundle $\mathcal{L}_{\infty }$ over $M_{\infty
}^{o}$ such that $F_{\infty }$ is the curvature of a Yang-Mills connection $%
A_{\infty }$, and, for $i\gg 1$,
\begin{equation*}
\mathcal{L}_{\infty}=\Phi _{K_{a}}^{i,-1}\mathcal{L}_{i}.
\end{equation*}

The remainder of the proof is the same as arguments in Theorem C in \cite%
{An1} and Theorem 2.6 in \cite{An2}. For the convenience of readers we
sketch the proof here, and refer to the two papers for more details.

Let $M_{\infty }=M_{\infty }^{o}\cup \{p_{k}\}_{k=1}^{N}$ be the
metric completion of $M_{\infty }^{o}.$ By using a blow up argument
in \cite{An1},  $M_{\infty }$ is a connected orbifold with a finite
number of singular points $\{p_{k}\}_{k=1}^{N}\subset M_{\infty },$
and the metric can be extended to a $C^{0}$ orbifold Riemannian
metric on $M_{\infty }.$ Furthermore, each singular point $p_{k}$
has a neighborhood homeomorphic to the cone $C\left( S^{n-1}/\Gamma
_{k}\right) ,$ with $\Gamma _{k}$ a finite subgroup of $O\left(
n\right) .$ \
\end{proof}

\bigskip \textbf{Acknowledgement. }I am grateful to Professor F.Q. Fang for
his brilliant guidance and to Professor Y.G. Zhang for a lot of stimulating
conversations. 


\bigskip

\begin{thebibliography}{99}
\bibitem{An1} M.T.Anderson, \emph{Ricci curvature bounds and Einstein
metrics on compact manifolds, } \textrm{J. Am. Math. Soc. 2, (1989)
455--490.}

\bibitem{An2} M.T.Anderson, \emph{Convergence and rigidity of manifolds
under Ricci curvature bounds, }\textrm{Invent. math. 102, (1990)
429--445.}

\bibitem{An3} M.T.Anderson, \emph{Einstein metrics and metrics with bounds
on Ricci curvature,} \textrm{Proceedings of International Congress
of Mathematicians, Vol.1,2(Z\"{u}rich, 1994) Birkh\"{a}user,
443--452.}

\bibitem{BKN} S.Bando, A.Kasue, H.Nakajima, \emph{On a construction of
coordinates at infinity on manifolds with fast curvature decay and
maximal volume growth,} \textrm{Inventiones Math. 97 (1989)
313--349.}

\bibitem{Be} A.L.Besse, \emph{Einstein manifolds,} \textrm{Ergebnisse der
Math. Springer-Verlag, Berlin-New York (1987).}

\bibitem{Ch} J.Cheeger, \emph{Finiteness theorems for Riemannian manifolds, }%
\textrm{Amer. J. Math. 92, (1970) 61--74.}



\bibitem{CT} J.Cheeger, G.Tian, \emph{Curvature and injective radius
estimates for Einstein 4-manifolds, }\textrm{J. Amer. Math. Soc.
19(2) (2006) 487--525.}



\bibitem{Che} S.S.Chern, \emph{Circle bundles,} \textrm{Lecture Notes in
Mathematics, 597(1977), Springer Berlin/Heidelberg.}

\bibitem{Go} M. Gromov, {\em Metric structures for Riemannian and
non-Riemannian spaces}, Birkh\"{a}user 1999.

\bibitem{GT} D.Gilbarg, N.Trudinger, \emph{Elliptic partial differential
equations of second order, }\textrm{Springer-Verlag, New York, (1977).}

\bibitem{GW} R.Greene, H.Wu, \emph{Lipschitz convergence of Riemannian
manifolds,} \textrm{Pacific J. Math. 131, (1988) 119--141.}

\bibitem{Hi} N.Hitchin, \emph{On Compact four-Dimensional Einstein Manifolds}%
, \textrm{J. Differential Geometry, Vol.9 (1974) 435--441.}

\bibitem{LeB} C.Lebrun, \emph{The Einstein-Maxwell Equations,Extremal K\"{a}%
hler Metrics,and Seiberg-Witten Theory}, \textrm{%
The Many Facets of Geometry: A Tribute to Nigel Hitchin, Oxford
University Press (2010).  }

\bibitem{LeB2} C.Lebrun, \emph{Einstein metrics, Four-Manifolds, and
Diffrerntial Topology, }\textrm{Surveys in Differential Geometry,
vol. VIII: Papers in Honor of Calabi, Lawson, Siu, and Uhlenbeck.
 International Press of Boston, (2003) 235--255.}

\bibitem{Pe} P.Peterson, \emph{Riemannian Geometry, }\textrm{Science Press,
(2007).}

\bibitem{St} J.Streets, \emph{Ricci Yang-Mills Flow, }\textrm{Ph.D.
Dissertation at The Department of Mathematics of Duke University, (2007).}

\bibitem{Ti} G.Tian, \emph{On Calabi's conjecture for complex surfaces with
positive first Chern class.} \textrm{Invent. math. 101 (1990),
101--172.}

\bibitem{Yo} A.Young, \emph{Modified Ricci flow on a Principal Bundle, }%
\textrm{Ph.D.  Dissertation at The University of Texas, (2008).}
\newline
\newline
\newline
\newline
\newline
ADDRESS: Department of Mathematics, Capital Normal University,
Beijing, China \newline  E-MAIL: liupinjinshi@gmail.com
\end{thebibliography}
\end{document}